\documentclass[12pt]{amsart}

\usepackage{amsmath,amsfonts,amssymb}

\theoremstyle{plain}
\newtheorem{lemma}{Lemma}[section]
\newtheorem{theorem}[lemma]{Theorem}

\theoremstyle{definition}

\numberwithin{equation}{section}
\begin{document}

\newcommand{\ZZ}{\mathbb{Z}}
\newcommand{\ZZd}{\mathbb{Z}^{d}}
\newcommand{\RR}{\mathbb{R}}
\newcommand{\RRd}{\mathbb{R}^{d}}
\newcommand{\PP}{\mathbb{P}}
\newcommand{\QQ}{\mathbb{Q}}
\newcommand{\EE}{\mathbb{E}}
\newcommand{\BB}{\mathbb{B}}
\newcommand{\mA}{\mathcal{A}}
\newcommand{\mB}{\mathcal{B}}
\newcommand{\mC}{\mathcal{C}}
\newcommand{\mD}{\mathcal{D}}
\newcommand{\mE}{\mathcal{E}}
\newcommand{\mF}{\mathcal{F}}
\newcommand{\mG}{\mathcal{G}}
\newcommand{\mI}{\mathcal{I}}
\newcommand{\mJ}{\mathcal{J}}
\newcommand{\mL}{\mathcal{L}}
\newcommand{\mkN}{\mathfrak{N}}
\newcommand{\mR}{\mathcal{R}}
\newcommand{\mS}{\mathcal{S}}
\newcommand{\mT}{\mathcal{T}}
\newcommand{\mU}{\mathcal{U}}
\newcommand{\mV}{\mathcal{V}}
\newcommand{\mW}{\mathcal{W}}
\newcommand{\bks}{\backslash}
\newcommand{\half}{\frac{1}{2}}
\newcommand{\fN}{\frac{1}{N}}
\newcommand{\tZ}{\tilde{Z}}
\newcommand{\tJ}{J_1}
\newcommand{\tK}{\tilde{K}}
\newcommand{\tf}{\tilde{f}}
\newcommand{\tE}{\tilde{E}}
\newcommand{\tL}{\tilde{L}}
\newcommand{\tP}{\tilde{P}}
\newcommand{\tS}{\tilde{\tau}}
\newcommand{\tG}{\tilde{\Gamma}}
\newcommand{\hd}{\hat{\delta}}
\newcommand{\hP}{\hat{P}}
\newcommand{\hQ}{\hat{Q}}
\newcommand{\hZ}{\hat{Z}}
\newcommand{\hH}{\hat{H}}
\newcommand{\hm}{\hat{\mu}}
\newcommand{\ophi}{\overline{\varphi}}
\newcommand{\olm}{\overline{m}}
\newcommand{\tphi}{\tilde{\varphi}}
\newcommand{\oF}{\overline{F}}
\newcommand{\bx}{\mathbf{x}}
\newcommand{\bV}{\mathbf{V}}
\newcommand{\by}{\mathbf{y}}
\newcommand{\bs}{\mathbf{s}}
\newcommand{\bS}{\mathbf{S}}
\newcommand{\bbL}{\mathbb{L}}
\newcommand{\ep}{\epsilon}
\newcommand{\hp}{\hat{\varphi}}

\title[Subgaussian Rates of Convergence in FPP]
{Subgaussian Rates of Convergence of Means in Directed First Passage Percolation}
\author{Kenneth S. Alexander}
\address{Department of Mathematics KAP 108\\
University of Southern California\\
Los Angeles, CA  90089-2532 USA}
\email{alexandr@usc.edu}
\thanks{Research supported by NSF grant DMS-0804934.}

\keywords{first passage percolation, time constant, subadditivity}
\subjclass[2010]{Primary 60K35}
\date{\today}

\begin{abstract}
We consider directed first passage percolation on the integer lattice, with time constant $\mu$ and passage time $a_{0n}$ from the origin to $(n,0,\dots,0)$.  It is shown that under certain conditions on the passage time distribution, $Ea_{0n} - n\mu = O(n^{1/2}(\log\log n)/\log n)$.
\end{abstract}
\maketitle

\section{Introduction} \label{S:intro}
The question of the order of the fluctuations (i.e. the standard deviation) for the passage time over a distance $n$ in first passage percolation (FPP) has attracted considerable interest.  FPP is believed to be one of a family of models in which these fluctuations are of order $n^{1/3}$ in two dimensions, with a Tracy-Widom distributional limit under corresponding scaling.  In three dimensions the conjectured order is $n^\chi$ with $\chi$ near 1/4.  Little is known rigorously in this regard, however, and even physicists do not agree on the behavior of the exponent $\chi$ as the dimension becomes large.  For directed last passage percolation in two dimensions with geometrically distributed passage times, the $n^{1/3}$ scale and the Tracy-Widom distributional limit were proved in \cite{Jo01}, but the ``exact solution'' methods used there do not seem conducive to use with more general passage time distributions.  (A more probabilitistic proof for geometrically distributed passage times is in \cite{CG06}.)  For longest increasing subsequences, which can be transformed into an FPP-like problem, analogous results were established in \cite{BDJ99}.  For first passage percolation, fluctuations are known to be subgaussian, but barely---Benjamini, Kalai and Schramm \cite{BKS03} showed that for Bernoulli-distributed passage times, the fluctuations are $O((n/\log n)^{1/2})$, and Bena\"im and Rossignol \cite{BR08} proved exponential bounds, on the scale $(n/\log n)^{1/2}$, for the fluctuations under more general passage time distributions, improving Kesten's exponential bounds \cite{Ke93} which are on scale $n^{1/2}$.

Here by an exponential bound on scale $c_n$ we mean that, letting $a_{0n}$ be the passage time from the origin to $(n,0,...,0)$, one has $P(|a_{0n} - Ea_{0n}| > tc_n)$ decaying exponentially in $t$, uniformly in $n$.

Besides deviations from the mean, it is of interest to understand deviations from the asymptotic value $n\mu$, where $\mu$ is the time constant.  This means that, letting $a_{0n}$ be the passage time from the origin to $(n,0,...,0)$, we need bounds for $Ea_{0n} - n\mu$.  Subadditivity of $Ea_{0n}$ ensures that this difference is nonnegative.  The best known bound, in \cite{Al93}, is of order $n^{1/2}\log n$, obtained with the help of Kesten's scale-$n^{1/2}$ exponential bounds on fluctuations.  In \cite{Al97}, analogous results are established for a wider class of models, and the proofs there show that for FPP in two or three dimensions, if one has an exponential bound on some scale $c_n$, then 
\begin{equation} \label{cnbound}
  Ea_{0n} - n\mu = O(c_n\log n).
  \end{equation}
For the proven scale $c_n = (n/\log n)^{1/2}$ from \cite{BR08}, though, this approach obviously does not lead to a subgaussian (i.e. $o(n^{1/2})$) bound on $Ea_{0n} - n\mu$.  Our aim here is to establish such a subgaussian bound by other methods, showing that in the directed case, for a reasonably wide class of passage time distributions, we have
\begin{equation} \label{rate}
  Ea_{0n} - n\mu = O\left( \frac{n^{1/2}\log \log n}{(\log n)^{1/2}} \right).
  \end{equation}
  
To formalize things, we write sites of $\ZZ^{d+1}$ as $(n,x)$ with $n \in \ZZ, x \in \ZZ^d$.  Let $\bbL^{d+1}$ be the even sublattice of $\ZZ^{d+1}$:
\[
  \bbL^{d+1} = \{(n,x) \in \ZZ^{d+1}: n+x_1 + \dots + x_d \text{ is even}\}.
  \]
Sites $(n,x)$ and $(n+1,y)$ in $\bbL^{d+1}$ are \emph{adjacent} if the Euclidean distance $|y-x|_2=1$.  A \emph{bond} $\langle (n,x),(n+1,y) \rangle$ is a pair of adjacent sites, and $\BB_{d+1}$ denotes the set of all bonds in $\bbL^{d+1}$.  We assign i.i.d.~nonnegative \emph{passage times} $\{\omega_b: b \in \BB_{d+1}\}$ to the bonds in $\BB_{d+1}$, with distribution $\nu$.  A \emph{(directed lattice) path} from $(n,x)$ to $(n+m,y)$ in $\bbL^{d+1}$ is a sequence of sites $((n,x^{(0)}),(n+1,x^{(1)}),\dots,(n+m,x^{(m)}))$ in which consecutive sites are adjacent, or it may be viewed as the corresponding sequence of bonds; which one should be clear from the context.  The \emph{passage time} for a path $\gamma$ is 
\[
  T(\gamma) = \sum_{b \in \gamma} \omega_b.
  \]
For sites $(n,x),(n+m,y)$ in $\bbL^{d+1}$, we then define
\[
  T((n,x),(n+m,y)) = \min\{T(\gamma):  \gamma \text{ is a lattice path from $(n,x)$ to } (n+m,y)\}
  \]
and for $n$ even,
\[
  a_{0n} = T((0,0),(n,0)).
  \]
A \emph{geodesic} is a path which achieves this minimum; the geodesic is a.s.~unique when $\omega_b$ has a continuous distribution.  The \emph{time constant} $\mu=\mu(\nu,d)$ is 
\begin{equation} \label{mudef}
  \mu = \lim_n \frac{Ea_{0n}}{n} = \inf_n \frac{Ea_{0n}}{n} = \lim_n \frac{a_{0n}}{n}\ \text{a.s.},
  \end{equation}
the existence of the first limit (taken through even $n$) being a consequence of subadditivity and the second being a consequence of Kingman's subadditive ergodic theorem \cite{Ki68}.  

In order to use the exponential bound of \cite{BR08}, we need to assume that $\nu$ is a nearly gamma distribution, as defined in \cite{BR08}.   To state the definition, let $F$ and $\Phi$ be the d.f.'s of $\nu$ and the standard normal distribution, respectively.   Assume $\nu$ has a density $f$, and let $\varphi$ be the density of the standard normal.  Define $I = \{t \geq 0: f(t)>0\}$ and on $I$ define the function
\[
  \Upsilon(y) = \frac{ \varphi \circ \Phi^{-1}(F(y)) }{ f(y) }.
  \]
The distribution $\nu$ is said to be \emph{nearly gamma} if $I$ is an interval, $f$ restricted to $I$ is continuous, and for some positive $A$,
\[
  \Upsilon(y) \leq A\sqrt{y} \quad \text{ for all } y \in I.
  \]
For a standard normal variable $\xi$, $F^{-1}(\Phi(\xi))$ has distribution $\nu$, and the nearly gamma property ensures that the map $F^{-1} \circ \Phi$ is nice enough that a log Sobolev inequality for $\xi$ translates into useful information about $\nu$; see \cite{BR08}.  Most common continuous distributions are nearly gamma---a sufficient condition for the property, from \cite{BR08}, is as follows.  Let $a<b$ be the (possibly infinite) endpoints of $I$.  Suppose that for some $\alpha>-1$,
\[
  \frac{f(x)}{(x-a)^\alpha} \quad\text{stays bounded away from 0 and $\infty$ as } x \searrow a,
  \]
and either (i) $b<\infty$ and for some $\beta>-1$,
\[
  \frac{f(x)}{(b-x)^\beta} \quad\text{stays bounded away from 0 and $\infty$ as } x \nearrow b,
  \]
or (ii) $b=\infty$ and 
\[
  \frac{f(x)}{ \int_x^\infty f(u)\ du} \quad\text{stays bounded away from 0 and $\infty$ as } x \nearrow \infty.
  \]
Then $\nu$ is nearly gamma.  

We can now state our main result.  

\begin{theorem}\label{main}
  Suppose $\nu$ is nearly gamma and $\int e^{tx}\ \nu(dx) < \infty$ for some $t>0$.  Then for even $n$,
  \begin{equation} \label{rate2}
  Ea_{0n} - n\mu = O\left( \frac{n^{1/2}\log \log n}{(\log n)^{1/2}} \right).
  \end{equation}
\end{theorem}

This bound is likely not sharp---an exponential bound on the expected scale $n^{1/3}$ would lead to \eqref{cnbound} with $c_n=n^{1/3}$, for example.

\section{Proof of Theorem \ref{main}} \label{S:proof}
We consider paths $\{(i,x^{(i)}):0\leq i \leq kn\}$ from $(0,0)$ to $(kn,0)$, for some block length $n$ to be specified and $k \geq 1$.  In general we take $n$ sufficiently large, and then take $k$ large, depending on $n$; we tacitly take $n$ to be even, throughout.  The \emph{simple skeleton} of such a path is $\{(jn,x^{(jn)}):0 \leq j \leq k\}$.  The number of possible simple skeletons is at most $(2n)^{dk}$, so if we have a probability associated to each skeleton which is bounded by some quantity $e^{-r_nk}$, then in order for the corresponding bound $(2n)^{dk}e^{-r_nk}$ on the sum to be small, 
we should have $r_n$ at least of order $\log n$.  This principle, in a different context, underlies the $\log n$ factor which appears in \eqref{cnbound}.  In \eqref{rate2} we essentially want to replace this $\log n$ with $\log \log n$.  Since $(2n)^{dk}$ is unacceptably large for this, it requires entropy reduction---the replacement of a sum over all simple skeletons with the sum over a small subclass which can be shown to approximate the full class of simple skeletons appropriately.  This entropy reduction is the main theme of our proof.

Our main technical tool is the following.  It is proved in \cite{BR08} for undirected FPP, but the proof for the directed case is the same.  The result there carries the additional hypothesis that $t \leq m$, but that is readily removed---the proof is in Section \ref{S:lemmas}.

\begin{theorem} \label{BR} 
  Suppose $\nu$ is nearly gamma and $\int e^{tx}\ \nu(dx) < \infty$ for some $t>0$.  Then there exist $C_1,C_2>0$ such that for 
  all $m \geq 2$ and $(m,x) \in \bbL^d$ with $|x|_1 \leq m$,
  \begin{equation} \label{expbound}
  P\left( \big|T((0,0),(m,x)) - ET((0,0),(m,x)) \big| > t\left(\frac{m}{\log m} \right)^{1/2} \right) \leq C_1 e^{-C_2t}.
  \end{equation}
\end{theorem}

If the exponential bound \eqref{expbound} could be improved to some scale $c_m \ll (m/\log m)^{1/2}$, then the rate on the right side of \eqref{rate2} could improved to $O(c_n\log\log n)$, with the proof virtually unchanged.

Let $H_m=(\{m\} \times \ZZ^d) \cap \bbL^d$.
A routine extension of Theorem \ref{BR} to sums of passage times is as follows.  We postpone the proof to Section \ref{S:lemmas}.

\begin{lemma} \label{sums}
Let $\nu,C_1,C_2$ be as in Theorem \ref{BR}, let $n_{\max} \geq 1$ and let $0 \leq s_1<t_1 \leq s_2 < t_2 < \dots \leq s_r < t_r$ with $t_j-s_j \leq n_{\max}$ for all $j \leq r$.  For each $j \leq r$ let $(s_j,y^{(j)}) \in H_{s_j}$ and $(t_j,z^{(j)}) \in H_{t_j}$, and let $T_j = T((s_j,y^{(j)}),(t_j,z^{(j)}))$.  Then for $C_3=C_2/(C_1+1)$, for $a>0$,
\begin{equation}
  P\left( \sum_{j=1}^r |T_j - ET_j| > 2a\right) \leq 2^{r+1} \exp\left( -C_3a\left( \frac{\log n_{\max}}{n_{\max}} \right)^{1/2} \right).
  \end{equation}
\end{lemma}

For $(m,x) \in \bbL^d$ define the excess mean $s(m,x)$ by
\[
  s(m,x) = ET((0,0),(m,x)) - m\mu.
  \]
Nonnegativity of $s$ follows from \eqref{mudef}, symmetry and subadditivity:  
\[
  s(m,x) = \frac{s(m,x) + s(m,-x)}{2} \geq \frac{s(2m,0)}{2} \geq 0,
  \]
and clearly $s(n,\cdot)$ is invariant under reflections across axes and under permutations of coordinates.  Let
\begin{equation} \label{slowlyvar}
  \psi(m) = \frac{C_3(\log m)^{1/2}}{\log \log m}, \quad \theta(m) = (\log m)^{3/2}, \quad \text{and} \quad \varphi(m) = \lfloor (\log m)^2 \rfloor.
  \end{equation}
For our designated block length $n$, for $x \in \ZZ^d$ with $(n,x) \in \bbL^d$, we say the transverse increment $x$ is \emph{excessive} if $s(n,x) > n^{1/2}\theta(n)$, and \emph{unexcessive} otherwise.  Note the dependence on $n$ is suppressed in this terminology.  From \cite{Al93} we know that there exists $C_0$ such that
\begin{equation} \label{sbound}
  s(m,0) \leq C_0m^{1/2}\log m \quad \text{for all even } m \geq 2,
  \end{equation}
from which it follows readily that provided $n$ is large, all $x$ with $|x| \leq n^{1/2}\log n$ are unexcessive.  (As with Theorem \ref{BR}, \eqref{sbound} is stated in \cite{Al93} for undirected FPP, but the proof for the directed case is the same.)  Let
\[
  h_n = \max\{|x|_\infty: x \text{ is unexcessive}\},
  \]
and let $x^*$ be an unexcessive site with $|x^*|=h_n$ and $x_1^* = |x^*|_\infty$.  We wish to coarse-grain on scale $u_n = 2\lfloor h_n/2\varphi(n) \rfloor$; note $u_n$ is an even integer.
A \emph{coarse-grained} (or \emph{CG}) \emph{point} is a point of form $(jn,x^{(jn)})$ with $j \geq 0$ and $x^{(jn)} \in u_n\ZZ^d$.
A \emph{coarse-grained} (or \emph{CG}) \emph{skeleton} is a simple skeleton $\{(jn,x^{(jn)}):0 \leq j \leq k\}$ consisting entirely of CG points.  
  
By a \emph{CG path} we mean a path from $(0,0)$ to $(kn,0)$ for which the simple skeleton is a CG skeleton.  We proceed as follows.  Given the geodesic $\Gamma$ from $(0,0)$ to $(kn,0)$, we would like to reroute $\Gamma$ into a CG path $\tG$ by changing it only within distance $n_1 \leq 6dn/\varphi(n)$ of each hyperplane $H_{jn}$, and in such a way that with high probability the passage times of $\Gamma$ and $\tG$ are close.  (Here $n_1$ is an even integer to be specified, cf.~Lemma \ref{cutpaste}.)  Once this is done, we would hope to achieve entropy reduction by effectively considering only CG paths.  What we actually do is slightly different---$\tG$ is not necessarily a true CG path but rather a sequence of segments each (with certain exceptions) having CG points as endpoints,
with one segment not required to connect to the next one.  But the effect is the same.

To this end, given a site $w=(jn\pm n_1,y^{(jn\pm n_1)}) \in H_{jn\pm n_1}$, let $z^{(jn)}$ be the site in $u_n\ZZ^d$ closest to $y^{(jn\pm n_1)}$ in $\ell^1$ norm (breaking ties by some arbitrary rule), and let $\pi_{jn}(w) = (jn,z^{(jn)})$, which may be viewed as the projection into $H_{jn}$ of the CG approximation to $w$ within the hyperplane $H_{jn\pm n_1}$.  Given the geodesic $\Gamma=\{(i,x^{(i)})\}$ from $(0,0)$ to $(kn,0)$, define points
\[
  d_j = d_j(\Gamma) = (jn,x^{(jn)}), \quad 0 \leq j \leq k,
  \]
\[
  e_j = (jn+n_1,x^{(jn+n_1)}), \quad 0 \leq j \leq k-1,
  \]
\[
  f_j = (jn-n_1,x^{(jn-n_1)}), \quad 1 \leq j \leq k.
  \]
We say a \emph{sidestep} occurs in block $j$ if either
\[
  |x^{((j-1)n+n_1)} - x^{((j-1)n)}|_\infty > h_n \quad \text{or} \quad |x^{(jn)} - x^{(jn-n_1)}|_\infty > h_n.
  \]
Let 
\[
  \mE_{ex} = \mE_{ex}(\Gamma) = \{1 \leq j \leq k:  x^{(jn)} - x^{((j-1)n)} \text{ is excessive}\},
  \]
\[
  \mE_{side} = \mE_{side}(\Gamma) = \{1 \leq j \leq k:  j \notin \mE_{ex} \text{ and a sidestep occurs in block } j\},
  \]
\[
  \mE = \mE_{ex} \cup \mE_{side}
  \]
and let
\[
  e_{j-1}' = \pi_{(j-1)n}(e_{j-1}), \quad  f_j' = \pi_{jn}(f_j), \quad j \notin \mE.
  \]
Define the tuples
\[
  \mR_j = \mR_j(\Gamma) = \begin{cases} (d_{j-1},d_j) &\text{if } j \in \mE_{ex},\\ (d_{j-1},e_{j-1},f_j,d_j) &\text{if } j \in \mE_{side},\\
    (e_{j-1}',f_j') &\text{if } j \notin \mE, \end{cases}
  \]
and define the \emph{CG-approximate skeleton} of $\Gamma$ to be
\[
  S_{CG}(\Gamma) = \{ \mR_j: 1 \leq j \leq k\}.
  \]
Note $\mE_{side}(\Gamma)$ is a  function of $S_{CG}(\Gamma)$.
Let $\mC$ denote the class of all possible CG-approximate skeletons of paths from $(0,0)$ to $(kn,0)$, and for $B \subset \{1,\dots,k\}$ let 
$\mC_B$ denote the class of all CG-approximate skeletons in $\mC$ with $\mE=B$.
$S_{CG}(\Gamma)$ is thus a sequence of tuples of sites; in each tuple the first and last sites will be used as the endpoints of a path which approximates one of the $k$ segments of $\Gamma$.  

For a CG-approximate skeleton in $\mC_B$, if $\mR_1,\dots,\mR_{j-1}$ are specified and $j \in B$, then there are at most $(2h_nu_n^{-1}+1)^{2d} \leq (3\varphi(n))^{2d}$ possible values of $\mR_j$; if $j \notin B$ there are at most $(2n)^{4d}$.  
It follows that the number of CG-approximate skeletons satisfies 
\begin{equation} \label{CB}
  |\mC_B| \leq (3\varphi(n))^{2d(k-|B|)}(2n)^{4d|B|}.
  \end{equation}
Note that the factor $\varphi(n)$ in place of $n$ in \eqref{CB} represents the entropy reduction discussed above.  
For each $j \notin \mE$ let $g_j(\Gamma)$ be the path from $e_{j-1}'$ to $f_j'$, via $e_{j-1}$ and $f_j$, obtained 
from $\Gamma$ by replacing the segment of $\Gamma$ from $d_{j-1}$ to $e_{j-1}$ with the geodesic from $e_{j-1}'$ to $e_{j-1}$, and replacing the segment of $\Gamma$ from $f_j$ to $d_j$ with the geodesic from $f_j$ to $f_j'$.  
For $j \in \mE$ let $g_j(\Gamma)$ be the segment of $\Gamma$ from $d_{j-1}$ to $d_j$.  
Then define the collection of paths
\[
  \tG = \{g_j(\Gamma): 1 \leq j \leq k\}.
  \]
Note that for $j \notin \mE$, $g_j(\Gamma)$ has CG points as endpoints.  We define the passage time in the natural way:
\[
  T(\tG) = \sum_{j=1}^k T(g_j(\Gamma)).
  \]

Let $b_{nk} = \lfloor \frac{k\log\log n}{\log n} \rfloor$.  Taking $k$ sufficiently large (depending on $n$), we have
\begin{align} \label{twocases}
  \half &< P\left( T(\Gamma) < kn\mu + \frac{64dkn^{1/2}}{\psi(n)} \right) \notag \\
  &\leq P\left( T(\Gamma) < kn\mu + \frac{64dkn^{1/2}}{\psi(n)}, |\mE(\Gamma)| > b_{nk} \right) +
     P\left( T(\tG) - T(\Gamma) > \frac{64dkn^{1/2}}{\psi(n)} \right) \notag \\
  &\qquad + P\left( T(\tG) < kn\mu + \frac{128dkn^{1/2}}{\psi(n)}, |\mE(\Gamma)| \leq b_{nk} \right),
\end{align}
where the first inequality follows from the definition of $\mu$.
We will show that the first two probabilities on the right side are small, and then, since the third probability cannot also be small, we will see that we must have 
\[
  ET(v_{j-1}',w_j') < n\mu + \frac{256dn^{1/2}}{\psi(n)} \quad \text{for some } j \leq n,v_{j-1}' \in H_{(j-1)n}, w_j' \in H_{jn},
  \]
which in turn leads easily to \eqref{rate2}.  For the second probability on the right side of \eqref{twocases} we have
\begin{align} \label{skeldecomp}
  &P\left( T(\tG) - T(\Gamma) > \frac{64dkn^{1/2}}{\psi(n)} \right) \notag \\
  &\leq \sum_{B \subset \{1,\dots,k\}} \sum_{ \{R_j\} \in \mC_B } 
    P\left( T(\tG) - T(\Gamma) > \frac{64dkn^{1/2}}{\psi(n)}, \mE(\Gamma)=B, S_{CG}(\Gamma) = \{R_j\} \right).
\end{align}
Given $B \subset \{1,\dots,k\}$ and $\{R_j\} \in \mC_B$, for some $v_{j-1}',w_j',p_{j-1},v_{j-1},w_j,p_j$ we can express $\{R_j\}$ as
\[
  R_j = \begin{cases} (v_{j-1}',w_j'), &\text{if } j \notin B, \\ (p_{j-1},p_j) &\text{if $j \in B$ and $R_j$ is a 2-tuple}, \\
    (p_{j-1},v_{j-1},w_j,p_j) &\text{if $j \in B$ and $R_j$ is a 4-tuple}.
    \end{cases}
  \]
Then let $\mI_B(\{R_j\})$ be the set of all tuple collections $\{(p_{j-1},v_{j-1},w_j,p_j):j \notin B\}$ compatible with $\{R_j\}$ and having the following properties:
\[
  p_{j-1} \in H_{(j-1)n},\quad |p_{j-1} - v_{j-1}'|_\infty \leq 2h_n,
  \]
\[
  v_{j-1} \in H_{(j-1)n+n_1},\quad \pi_{(j-1)n}(v_{j-1}) = v_{j-1}',
\]
\[
  w_j \in H_{jn-n_1},\quad \pi_{jn}(w_j) = w_j',
  \]
\[
  p_j \in H_{jn}, \quad |p_j - w_j|_\infty \leq 2h_n.
  \]
Here by ``compatible with $\{R_j\}$'' we mean that the choice of $\{R_j\}$ specifies values $p_i$ for certain $i$
(specifically, for $i=j$ and $i=j-1$ for each $j \in B$); the tuple collections in  
$\mI_B(\{R_j\})$ must use these same values $p_i$.
Note that $p_{j-1},v_{j-1},w_j,p_j$ are possible values of the variables $d_{j-1},e_{j-1},f_j,d_j$, respectively, and
\begin{equation} \label{size}
  |\mI_B(\{R_j\})| \leq (4h_n+1)^{4d|B|} \leq n^{4d|B|}.
\end{equation}
Then let
\begin{align} \label{timedef}
  T_{skel}(\{R_j\}) &= \sum_{j \notin B} T(v_{j-1}',w_j') + \sum_{\substack{j \in B \\ R_j \text{ a 2-tuple}}} T(p_{j-1},p_j) \notag \\
  &\qquad + \sum_{\substack{j \in B \\ R_j \text{ a 4-tuple}}} \big[ T(p_{j-1},v_{j-1}) + T(v_{j-1},w_j) + T(w_j,p_j) \big].
\end{align}
Then
\begin{align} \label{skeldecomp2}
 P&\left( T(\tG) - T(\Gamma) > \frac{64dkn^{1/2}}{\psi(n)}, \mE(\Gamma)=B, S_{CG}(\Gamma) = \{R_j\} \right) \notag \\
 &\leq \sum_{ \substack{ \{(p_{j-1},v_{j-1},w_j,p_j):j \notin B\}\\ \in\ \mI_B(\{R_j\}) } }
   \Bigg\{ P\left( \sum_{j \notin B} \left[ T(v_{j-1}',v_{j-1}) - T(p_{j-1},v_{j-1})  \right] > \frac{32dkn^{1/2}}{\psi(n)} \right) \notag \\
 &\hskip 4.5cm + P\left( \sum_{j \notin B} \left[ T(w_j,w_j') - T(w_j,p_j))  \right] > \frac{32dkn^{1/2}}{\psi(n)} \right) \Bigg\}.
\end{align}

It may be noted that we have modified $\Gamma$ into $\tG$ so that we need only sum over CG-approximate skeletons, rather than over all simple skeletons, and we thereby achieved entropy reduction, but now in \eqref{skeldecomp2}, the number of terms in the sum effectively increases the entropy back to its original level.  However we have the advantage that all the path lengths involved in \eqref{skeldecomp2} are at most $n_1$, not $n$ as in the simple skeleton, so Lemma \ref{sums} will give a better bound, which will be sufficient to overcome the additional entropy.

To bound the first probability on the right side of \eqref{skeldecomp2}, we have
\begin{align} \label{leftend}
  P&\left( \sum_{j \notin B} \left[ T(v_{j-1}',v_{j-1}) - T(p_{j-1},v_{j-1})  \right] > \frac{32dkn^{1/2}}{\psi(n)} \right) \notag \\
  &\leq P\left( \sum_{j \notin B} T(v_{j-1}',v_{j-1}) > (k-|B|)n_1\mu + \frac{16dkn^{1/2}}{\psi(n)} \right) \notag \\
  &\qquad\qquad + P\left( \sum_{j \notin B} T(p_{j-1},v_{j-1}) < (k-|B|)n_1\mu -  \frac{16dkn^{1/2}}{\psi(n)} \right) \notag \\
  &\leq P\left( \sum_{j \notin B} T(\pi_{(j-1)n}(v_{j-1}),v_{j-1}) > (k-|B|)n_1\mu + \frac{16dkn^{1/2}}{\psi(n)} \right) \notag \\
  &\qquad \qquad + P\left( \sum_{j \notin B} \left[ T(p_{j-1},v_{j-1}) - ET(p_{j-1},v_{j-1}) \right] < -\frac{16dkn^{1/2}}{\psi(n)} \right).
\end{align}
To bound the first probability on the right side of \eqref{leftend}, we need an upper bound for $ET(\pi_{(j-1)n}(v_{j-1}),v_{j-1})$, supplied by the next lemma.  We postpone the proof to Section \ref{S:lemmas}.

\begin{lemma} \label{cutpaste}
Provided $n$ is sufficiently large, there exists an even integer $n_1 \leq 6dn/\varphi(n)$ such that for every $v \in H_{n_1}$ we have 
\[
  ET(\pi_0(v),v) \leq n_1\mu + \frac{8dn^{1/2}}{\psi(n)}.
  \]
\end{lemma}

It follows from Lemmas \ref{sums} (with $n_{\max}=n_1$) and \ref{cutpaste} that provided $n$ is large,
\begin{align} \label{insertmean}
  P&\left( \sum_{j \notin B} T(\pi_{(j-1)n}(v_{j-1}),v_{j-1}) > (k-|B|)n_1\mu + \frac{16dkn^{1/2}}{\psi(n)} \right) \notag \\
  &\leq P\left( \sum_{j \notin B} \left[ T(\pi_{(j-1)n}(v_{j-1}),v_{j-1}) - ET(\pi_{(j-1)n}(v_{j-1}),v_{j-1}) \right] 
    > \frac{8dkn^{1/2}}{\psi(n)} \right) \notag \\
  &\leq 2^{k+1} \exp\left( - \frac{4C_3dkn^{1/2}}{\psi(n)} \left( \frac{\log n_1}{n_1} \right)^{1/2} \right) \notag \\
  &\leq \exp\left( -k(d\varphi(n))^{1/2}\log\log n \right),
  \end{align}
and the same bound holds for the second probability in \eqref{leftend}, so \eqref{leftend} yields
\[
  P\left( \sum_{j \notin B} \left[ T(v_{j-1}',v_{j-1}) - T(p_{j-1},v_{j-1})  \right] > \frac{32dkn^{1/2}}{\psi(n)} \right)
    \leq 2\exp\left( -k(d\varphi(n))^{1/2}\log\log n \right).
    \]
This bound also holds for the last probability in \eqref{skeldecomp2}, which with \eqref{size} and \eqref{skeldecomp2} shows that
\begin{align*}
  P&\left( T(\tG) - T(\Gamma) > \frac{64dkn^{1/2}}{\psi(n)}, \mE(\Gamma)=B, S_{CG}(\Gamma) = \{R_j\} \right) \notag \\
  &\qquad \leq 4n^{4dk}\exp\left( -k(d\varphi(n))^{1/2}\log\log n \right).
  \end{align*}
Then from \eqref{CB} and \eqref{skeldecomp}, provided $n$ is large, 
\begin{align} \label{diffspeed}
  P\left( T(\tG) - T(\Gamma) > \frac{64dkn^{1/2}}{\psi(n)} \right) 
    &\leq 4 \cdot 2^k (2n)^{4dk} n^{4dk}\exp\left( -k(d\varphi(n))^{1/2}\log\log n \right) < \frac{1}{8}.
\end{align}

We turn now to the first probability on the right side of \eqref{twocases}.  Here we use the non-coarse-grained analogs of the $\mR_j$, given by
\[
  \mV_j = \mV_j(\Gamma) = \begin{cases} (d_{j-1},d_j) &\text{if } j \notin \mE_{side},\\ (d_{j-1},e_{j-1},f_j,d_j) &\text{if } j \in \mE_{side}, \end{cases}
  \]
and define
the \emph{augmented skeleton} of $\Gamma$ to be
\[
  \mS_{aug}(\Gamma) = \{\mV_j, 1 \leq j \leq k\}.  
  \]
Note that $\mE_{side}(\Gamma)$ and $\mE_{ex}(\Gamma)$ are functions of $\mS_{aug}(\Gamma)$.
Let $\mC_{aug}^+$ denote the class of all augmented skeletons for which $|\mE| \geq b_{nk}$.  For a given $\{V_j,1 \leq j \leq k\} \in \mC_{aug}^+$, for some $E_{side} \subset \{1,\dots,k\}$ and $p_j,v_j,w_j$ we can write $\{V_j,1 \leq j \leq k\} \in \mC_{aug}^+$ as
\[
  V_j = \begin{cases} (p_{j-1},p_j), &j \notin E_{side},\\ (p_{j-1},v_{j-1},w_j,p_j), &j \in E_{side}, \end{cases}
  \]
and we then define 
\[
  T_{skel}(\{V_j\}) = \sum_{j \notin E_{side}} T(p_{j-1},p_j) + \sum_{j \in E_{side}} \big[ T(p_{j-1},v_{j-1}) + T(v_{j-1},w_j) + T(w_j,p_j) \big],
  \]
where we take $p_0=0$.  (Note this definition is consistent with \eqref{timedef}.)  For $j \in E_{side}$, a sidestep occurs either from $p_{j-1}$ to $v_{j-1}$ or from $w_j$ to $p_j$.  In the former case, writing $v_{j-1}$ as $((j-1)n+n_1,y)$ for some $y \in \ZZ^d$, let $b_j = (jn,y)$.  From the definitions of ``sidestep'' and $h_n$ we have using \eqref{sbound} that
\begin{align} \label{sidestep}
  n\mu + n^{1/2}\theta(n) \leq ET(p_{j-1},b_j) &\leq ET(p_{j-1},v_{j-1}) + ET(v_{j-1},b_j) \notag \\
  &\leq ET(p_{j-1},v_{j-1}) + (n-n_1)\mu + C_0n^{1/2}\log n,
  \end{align}
so assuming $n$ is large,
\[
  ET(p_{j-1},v_{j-1}) \geq n_1\mu + \half n^{1/2}\theta(n),
  \]
and hence
\[
  E\big[ T(p_{j-1},v_{j-1}) + T(v_{j-1},w_j) + T(w_j,p_j) \big] \geq n\mu + \half n^{1/2}\theta(n).
  \]
The last bound holds similarly if the sidestep instead occurs from $w_j$ to $p_j$, so
\[
  ET_{skel}(\{V_j\}) \geq kn\mu + \half b_{nk}n^{1/2}\theta(n).
  \]
Therefore by Lemma \ref{sums} (with $n_{\max}=n$) and \eqref{slowlyvar}, provided $n$ is large,
\begin{align} \label{skeldecomp3}
  P&\left( T(\Gamma) < kn\mu + \frac{64dkn^{1/2}}{\psi(n)}, |\mE(\Gamma)| > b_{nk} \right) \notag \\
  &\leq \sum_{\{V_j\} \in \mC_{aug}^+} P\left( T_{skel}(\{V_j\}) - ET_{skel}(\{V_j\}) < - \half b_{nk}n^{1/2}\theta(n) \right) \notag \\
  &\leq (2n)^{3dk} 2^{k+1} \exp\left( -\frac{C_3k\theta(n)\log\log n}{16(\log n)^{1/2}} \right) \notag \\
  &< \frac{1}{8}.
\end{align}

Now we consider the third probability on the right side of \eqref{twocases}.  From \eqref{twocases}, \eqref{diffspeed} and \eqref{skeldecomp3} we have
\begin{align} \label{skeldecomp4}
  \frac{1}{4} &< \left( T(\tG) < kn\mu + \frac{128dkn^{1/2}}{\psi(n)}, |\mE(\Gamma)| \leq b_{nk} \right) \notag \\
  &\leq \sum_{\substack{B \subset \{1,\dots,k\} \\ |B| \leq b_{nk}}} \sum_{ \{R_j\} \in \mC_B } 
    P\left( T_{skel}(\{R_j\}) \leq kn\mu + \frac{128dkn^{1/2}}{\psi(n)} \right).
\end{align}
From here we proceed by contradiction.  Suppose that 
\begin{equation} \label{hypoth}
  s(n,x) \geq \frac{256dn^{1/2}}{\psi(n)} \quad \text{ for all } x \in H_n.
  \end{equation}
Then for the probability on the right of \eqref{skeldecomp4} we have from Lemma \ref{sums} (with $n_{\max}=n$) that
\begin{align} \label{tooslow}
  P&\left( T_{skel}(\{R_j\}) \leq kn\mu + \frac{128dkn^{1/2}}{\psi(n)} \right) \notag \\
  &\leq P\left( T_{skel}(\{R_j\}) - ET_{skel}(\{R_j\}) \leq -\frac{128dkn^{1/2}}{\psi(n)} \right) \notag \\
  &\leq 2^{3k+1} e^{ -64dk\log\log n},
  \end{align}
and then using \eqref{CB} and \eqref{skeldecomp4},
\begin{align} \label{skeldecomp5}
  \frac{1}{4} &< 2^{4k+1} (3\varphi(n))^{2dk}(2n)^{4db_{nk}} e^{-64dk\log\log n} \leq e^{-32dk\log\log n},
\end{align}
which is clearly false.  Therefore the inequality in \eqref{hypoth} must fail for some $x$, and then
\[
  s(2n,0) \leq s(n,x) + s(n,-x) \leq \frac{512dn^{1/2}}{\psi(n)}
  \]
provided $n$ is sufficiently large, which completes the proof of Theorem \ref{main}.  Note that \eqref{skeldecomp5} shows the purpose of entropy reduction---if $\varphi(n)$ were replaced by $n$, we would have to have $\log n$ in place of $\log\log n$ in the definition of $\psi(n)$.

\section{Proofs of Lemmas and Theorem \ref{BR}} \label{S:lemmas}
\begin{proof}[Proof of Theorem \ref{BR}]
As mentioned in the introduction, the theorem is proved in \cite{BR08} under the additional hypothesis that $t \leq m$, so we assume $t>m$.  Let $\mu_1=E(\omega_b)$, let $\gamma_{(m,x)}$ be a (nonrandom) path from $(0,0)$ to $(m,x)$ and write $T_{(m,x)}$ for $T((0,0),(m,x))$.  Let $m_1,m_2$ satisfy
\[
  \frac{m_1}{\log m_1} = 4\mu_1^2, \qquad m_2 = 2(m_1 \log m_1)^{1/2} \mu_1.
\]
We can always choose $C_1,C_2$ so that $C_1e^{-C_2t} \geq 1$ for all $t \leq m_2$, so we may assume $t>m_2$.  If $m \geq m_1$, then since $t>m$ we have
\begin{equation} \label{mu1bound}
  2\mu_1m \leq t\left( \frac{m}{\log m} \right)^{1/2},
\end{equation}
while if $m < m_1$ then \eqref{mu1bound} follows from $t>m_2$.  Either way, since $T_{(m,x)}\geq 0$ and $ET_{(m,x)} \leq m\mu_1$, we then have
\[
  T_{(m,x)} - ET_{(m,x)} \geq -t\left( \frac{m}{\log m} \right)^{1/2}.
\]
Further, by \eqref{mu1bound},
\begin{align}
  T_{(m,x)} - ET_{(m,x)} > t\left( \frac{m}{\log m} \right)^{1/2} &\implies T(\gamma_{(m,x)}) \geq T_{(m,x)} \geq 2\mu_1 m \notag \\
  &\implies T(\gamma_{(m,x)}) - \mu_1 m \geq \half T(\gamma_{(m,x)}) \geq \frac{t}{2} \left( \frac{m}{\log m} \right)^{1/2},
\end{align}
so, letting $I$ denote the rate function of the variable $\omega_b - E\omega_b$,
\begin{align} \label{LD}
  P\left( T_{(m,x)} - ET_{(m,x)} > t\left( \frac{m}{\log m} \right)^{1/2} \right) \leq \exp\left( -I\left( \frac{t}{2(m\log m)^{1/2}} \right) m \right).
\end{align}
Since $I$ is convex, there exists $C_4>0$ such that $I(x) \geq 2C_4x$ for all $x \geq 1/2$, and we have
\[
  \frac{t}{2(m\log m)^{1/2}} \geq \half \left( \frac{m}{\log m} \right)^{1/2} \geq \half,
\]
so by \eqref{LD},
\[
  P\left( T_{(m,x)} - ET_{(m,x)} > t\left( \frac{m}{\log m} \right)^{1/2} \right) \leq e^{-C_4t},
\]
which completes the proof.
\end{proof}

\begin{proof}[Proof of Lemma \ref{sums}]
We have
\begin{align} \label{markovbound}
  P\left( \sum_{j=1}^r (T_j - ET_j)_+ > a \right) 
    &\leq e^{- \lambda a} \prod_j E\exp\left( \lambda (ET_j - T_j)_+ \right) .
  \end{align}
By Theorem \ref{BR}, for each $j$, letting $r_j = t_j-s_j$, for $\lambda \leq C_3\left( \frac{\log n_{\max}}{n_{\max}} \right)^{1/2}$,
\begin{align} \label{expbound2}
  E\exp\left( \lambda (ET_j - T_j)_+ \right) &\leq 1 + \int_1^\infty P\left( ET_j - T_j > \frac{\log t}{\lambda} \right)\ dt \notag \\
  &\leq 1 + \int_1^\infty C_1\exp\left(- \frac{C_2\log t}{\lambda} \left(\frac{\log r_j}{r_j}\right)^{1/2} \right)\ dt \notag \\
  &= 1 + \frac{C_1}{\frac{C_2}{\lambda} \left(\frac{\log r_j}{r_j}\right)^{1/2} -1}.
\end{align}
Letting $\lambda = C_3\left( \frac{\log n_{\max}}{n_{\max}} \right)^{1/2}$, since $r_j \leq n_{\max}$ we obtain $E\exp\left( \lambda (ET_j - T_j)_+ \right) \leq 2$ and then by \eqref{markovbound},
\[
  P\left( \sum_{j=1}^r (T_j - ET_j)_+ > a \right) \leq 2^r \exp\left( -C_3a\left( \frac{\log n_{\max}}{n_{\max}} \right)^{1/2} \right).
  \]
The same bound holds for $P\left( \sum_{j=1}^r (ET_j - T_j)_+ > a \right)$.
\end{proof}
 
For the proof of Lemma \ref{cutpaste} we need the following.  
We say $(m,x)$ is \emph{slow} if 
\[
  s(m,x) \geq \frac{n^{1/2}}{\psi(n)},
  \]
and \emph{fast} otherwise.  A path $((l,x^{(0)}),(l+1,x^{(1)}),\dots,(l+m,x^{(m)}))$ is \emph{clean} if every increment $(t-s,x^{(t)}-x^{(s)})$ with $0 \leq s < t \leq m$ is fast.

Given a path $\gamma = \{(m,x^{(m)})\}$ from $(0,0)$ to $(n,x^*)$, let 
\[
  \tau_j = \tau_j(\gamma) = \min\{m: x_1^{(m)} = ju_n\}, \quad 1 \leq j \leq \varphi(n).
  \]
The \emph{climbing skeleton} of $\gamma$ is $C(\gamma) = \{(\tau_j,x^{(\tau_j)}): 1 \leq j \leq \varphi(n)\}$.
A \emph{climbing segment} of $\gamma$ is a segment of $\gamma$ from $(\tau_{j-1},x^{(\tau_{j-1})})$ to $(\tau_j,x^{(\tau_j)})$ for some $j$.  A climbing segment is \emph{short} if $\tau_j - \tau_{j-1} \leq 2n/\varphi(n)$, and \emph{long} otherwise.  (Note $n/\varphi(n)$ is the average length of the climbing segments in $\gamma$.)  Since the total length of $\gamma$ is $n$, there can be at most $\varphi(n)/2$ long climbing segments in $\gamma$, so there are at least $\varphi(n)/2$ short ones.  
\[
  \mJ_s(\gamma) = \{j \leq \varphi(n): \text{ the $j$th climbing segment of $\gamma$ is short} \},
  \]
\[
  \mJ_l(\gamma) = \{j \leq \varphi(n): \text{ the $j$th climbing segment of $\gamma$ is long} \},
  \]
\[
  J_l(\gamma) = \left( \cup_{j \in \mJ_l(\gamma)} (\tau_{j-1},\tau_j) \right) \cap \left\lfloor \frac{2n}{\varphi(n)} \right\rfloor \ZZ.
  \]
If no short climbing segment of $\gamma$ is clean, then for each $j \in \mJ_s(\gamma)$ there exist $\alpha_j(\gamma)<\beta_j(\gamma)$ in $[\tau_{j-1},\tau_j]$ for which the increment of $\gamma$ from $(\alpha_j,x^{(\alpha_j)})$ to $(\beta_j,x^{(\beta_j)})$ is slow.  (If $\alpha_j,\beta_j$ are not unique we make a choice by some arbitrary rule.)  
We can reorder the values $\{\tau_j, j \leq \varphi(n)\} \cup \{\alpha_j,\beta_j: j \in \mJ_s(\gamma)\} \cup J_l(\gamma)$ into a single sequence $\{\sigma_j, 1 \leq j \leq N(\gamma)\}$ with $\varphi(n) \leq N(\gamma) \leq 4\varphi(n)$, such that at least $\varphi(n)/2$ of the increments $(\sigma_j - \sigma_{j-1},x^{(\sigma_j)} - x^{(\sigma_{j-1})}), j \leq N(\gamma),$ are slow.  The \emph{augmented climbing skeleton} of $\gamma$ is then the sequence 
$A(\gamma) = \{(\sigma_j,x^{(\sigma_j)}): 1 \leq j \leq N(\gamma)\}$.

\begin{lemma} \label{existence}
Provided $n$ is large, there exists a path from $(0,0)$ to $(n,x^*)$ containing a short climbing segment which is clean.
\end{lemma}

Note that Lemma \ref{existence} is a purely deterministic statement, since the property of being clean does not involve the configuration $\{\omega_b\}$.  

Translating the segment obtained in Lemma \ref{existence} to begin at the origin, we obtain a path $\alpha^*$ from $(0,0)$ to some site $(m^*,y^*)$, with the following properties:
\begin{equation} \label{properties}
  m^* \leq \frac{2n}{\varphi(n)}, \quad y_1^* = u_n \quad \text{and $\alpha^*$ is clean}.
  \end{equation}
  
\begin{proof}[Proof of Lemma \ref{existence}]
Let $D$ denote the event that the geodesic $\gamma^*$ from $(0,0)$ to $(n,x^*)$ does not contain a short climbing segment which is clean.
We will show that $P(D)<1$, which is sufficient to prove the lemma.

Since $x^*$ is unexcessive, provided $n$ is large it follows from Theorem \ref{BR} that
\begin{equation} \label{bigprob}
  P\left(T^* < n\mu + 2n^{1/2}\theta(n)\right) > \half.
  \end{equation}
Let $\mA^*$ be the set of all augmented climbing skeletons of paths from $(0,0)$ to $(n,x^*)$, so $|\mA^*| \leq (2n)^{4(d+1)\varphi(n)}$ and 
\begin{align} \label{decomp}
  P&\left(D \cap \{T^* < n\mu + 2n^{1/2}\theta(n)\} \right) \\
  &\qquad \leq \sum_{\{(\sigma_j,x^{(\sigma_j)})\} \in \mA^*}
    P\left(D \cap \{ A(\gamma^*) = \{(\sigma_j,x^{(\sigma_j)})\} \} \cap \{T^* < n\mu + 2n^{1/2}\theta(n)\} \right) \notag \\
  &\qquad \leq \sum_{\{(\sigma_j,x^{(\sigma_j)})\} \in \mA^*}
    P\left( \sum_j T((\sigma_{j-1},x^{(\sigma_{j-1})}),(\sigma_j,x^{(\sigma_j)})) < n\mu + 2n^{1/2}\theta(n) \right). \notag
\end{align}
Here the sum over $j$ has at most $4\varphi(n)$ terms.
In each $\{(\sigma_j,x^{(\sigma_j)})\} \in \mA^*$ there are at least $\varphi(n)/2$ slow increments, so we have 
\[
  \sum_j ET((\sigma_{j-1},x^{(\sigma_{j-1})}),(\sigma_j,x^{(\sigma_j)})) > n\mu + \frac{\varphi(n)}{2} \frac{n^{1/2}}{\psi(n)}.
  \]
Then by \eqref{slowlyvar} and Lemma \ref{sums} (with $n_{\max} = 2n/\varphi(n)$), letting $T_j = T((\sigma_{j-1},x^{(\sigma_{j-1})}),(\sigma_j,x^{(\sigma_j)}))$, provided $n$ is large,
\begin{align} \label{sumbound}
  P\left( \sum_j T_j < n\mu + 2n^{1/2}\theta(n) \right) 
    &\leq P\left( \sum_j \left( T_j - ET_j \right) < -\frac{n^{1/2}\varphi(n)}{4\psi(n)} \right) \notag \notag \\
  &\leq 2^{4\varphi(n)+1} \exp\left(-\frac{C_3\varphi(n)^{3/2}(\log n)^{1/2}}{16 \psi(n)} \right) \notag \\
  &= 2^{4\varphi(n)+1} \exp\left(-\frac{1}{16}\varphi(n)^{3/2}\log\log n \right).
  \end{align}
Then by \eqref{bigprob} and \eqref{decomp}, provided $n$ is large,
\begin{align} \label{Dcap}
 P\left(D \cap \{T^* < n\mu + 2n^{1/2}\theta(n)\} \right) &\leq 
   (2n)^{3(d+1)\varphi(n)} 2^{4\varphi(n)+1} \exp\left(-\frac{1}{16}\varphi(n)^{3/2}\log\log n \right) \notag \\
 &\leq  \exp\left(-\frac{1}{32}\varphi(n)^{3/2}\log\log n \right) \notag \\
 &< P\left(T^* < n\mu + 2n^{1/2}\theta(n) \right),
\end{align}
which shows $P(D)<1$, as desired.
\end{proof}

\begin{proof}[Proof of Lemma \ref{cutpaste}]
We use the path $\alpha^*= \{(i,a^{(i)}), i \leq m^*\}$ satisfying \eqref{properties}.  
Write $v \in H_{n_1}$ as $(n_1,\hat{v})$ with $\hat{v} \in \ZZ^d$.
We may assume that $\pi_0(v)=0$, and then from symmetry that $0 \leq \hat{v}_l \leq u_n/2$ for all $l \leq d$.
Let $w \in \ZZ^d$ have all coordinates even, with $|w_l - \hat{v}_l| \leq 1$ for all $l \leq d$.
Define $\zeta$ and $\xi_j$ on $\ZZ^d$ by
\[
  \zeta(i,x_1,\dots,x_d) = (i,x_1,-x_2,\dots,-x_d), 
  \]
\[
  \xi^j(i,x_1,\dots,x_d) = (i,x_j,x_2,\dots,x_{j-1},x_1,x_{j+1},\dots,x_d), \quad 1 \leq j \leq d.
  \]
Then $\alpha^*$ and $\zeta \circ \alpha^*$ are clean paths, the latter the reflection of the former through the plane $\{x:x_2=\cdots=x_d=0\}$, with both paths ending at sites with first transverse coordinate $u_n$.  When we add the two paths together we obtain a ``path'' which stays in that plane:
\[
  \alpha^*(i) + \zeta \circ \alpha^*(i) = (2i,2a_1^{(i)},0,\dots,0),  \quad i \leq m^*,
  \]
while composing both maps with $\xi^j$ interchanges the roles of the first and $j$th transverse coordinates:
\begin{equation} \label{symmpath}
  \xi^j \circ \alpha^*(i) + \xi^j \circ \zeta \circ \alpha^*(i) = (2i,0,\dots,0,2a_1^{(i)},0,\dots,0),  \quad i \leq m^*, j \leq d.
  \end{equation}
We put ``path'' in quotes because \eqref{symmpath} is not a true path, as consecutive sites are not adjacent; nonetheless we refer to it as the \emph{jth symmetrized path}.
Since $a_1^{(m^*)} = u_n$ and $w_j \in [0,u_n+1]$ is even for all $j \leq d$, we can define $i_j = \min\{i \leq m^*: 2a_1^{(i)} = w_j\}$, which is the time when the $j$th symmetrized path reaches height $w_j$ in the $j$th transverse coordinate.  We then have
\begin{equation} \label{cancel}
  \xi^j \circ \alpha^*(i_j) + \xi^j \circ \zeta \circ \alpha^*(i_j) = (2i_j,0,\dots,0,w_j,0,\dots,0),  \quad j \leq d.
  \end{equation}
Now define 
\[
  \eta(i,x) = (i,-x), \quad (i,x) \in \ZZ^{d+1}. 
  \]
Then $\eta \circ \alpha^*$ is clean, and 
\begin{equation} \label{horiz}
  \alpha^*(m^*-i_j) + \eta \circ \alpha^*(m^*-i_j) = (2(m^*-i_j),0,\dots,0),  \quad j \leq d.
  \end{equation}
If we append the ``horizontal path'' $\alpha^* + \eta \circ \alpha^*$ with endpoint \eqref{horiz} to the $j$th symmetrized path with endpoint \eqref{cancel}, we obtain a ``path'' of fixed length $2m^*$.  Summing the increments of all the symmetrized and horizontal paths we obtain
\[
  \sum_{j=1}^d \left[ \xi^j \circ \alpha^*(i_j) + \xi^j \circ \zeta \circ \alpha^*(i_j) + \alpha^*(m^*-i_j) + \eta \circ \alpha^*(m^*-i_j) \right]
    = (2dm^*,w),
  \]
which expresses $(2dm^*,w)$ as a sum of $4d$ fast increments.  Since $|\hat{v}-w| \leq d$, $(2d,\hat{v}-w)$ is fast provided $n$ is large, so $(2d(m^*+1),\hat{v})$ is given as a sum of $4d+1$ fast increments.  It follows that 
\[
  s(2d(m^*+1),\hat{v}) \leq \frac{(4d+1)n^{1/2}}{\psi(n)}.
  \]
Taking $n_1=2d(m^*+1)$ completes the proof.
\end{proof}

In Lemmas \ref{cutpaste} and \ref{existence} we have effectively bounded the difference between a subadditive quantity, in this case of form $ET((0,0),(n_1,u))$, and its asymptotic approximation $n_1\mu$ by expressing $(n_1,u)$ as a sum of a bounded number of increments of a clean path.  Earlier uses of a similar idea in other contexts appear in \cite{Al90} and \cite{Al01}.

\end{document}